\documentclass[eqsecnum]{revtex4-1}
\usepackage{graphicx}
\usepackage{amsmath}
\usepackage{amssymb}
\usepackage{amscd}
\usepackage{amsthm}
\usepackage{hyperref}
\usepackage{color}
\usepackage{enumerate}
\usepackage{caption}
\usepackage{subcaption}
\usepackage{adjustbox}
\usepackage[top=1in, bottom=1in, left=1in, right=1in]{geometry}
\usepackage{natbib}

\newtheorem{theorem}{Theorem}
\newtheorem{lemma}[theorem]{Lemma}
\newtheorem{corollary}[theorem]{Corollary}

\newcommand{\eps}{\varepsilon}

\begin{document}

\title{Canard explosion and relaxation oscillation in planar, piecewise-smooth, continuous systems}
\date{\today}

\author{Andrew Roberts}
\affiliation{Department of Mathematics, Cornell University}

\begin{abstract}
Classical canard explosion results in smooth systems require the vector field to be at least $C^3$, since canard cycles are created as the result of a Hopf bifurcation.  The work on canards in nonsmooth, planar systems is recent and has thus far been restricted to piecewise-linear or piecewise-smooth Van der Pol systems, where an extremum of the critical manifold arises from the nonsmoothness.  In both of these cases, a canard (or canard-like) explosion may be created through a nonsmooth bifurcation as the slow nullcline passes through a corner of the critical manifold.  Additionally, it is possible for these systems to exhibit a super-explosion bifurcation where the canard explosion is skipped. This paper extends the results to more general piecewise-smooth systems, finding conditions for when a periodic orbit is created through either a smooth or nonsmooth bifurcation.  In the case the bifurcation is nonsmooth, conditions are found determining whether the bifurcation is a super-explosion or canards are created.  
\end{abstract}

\maketitle


\section{Introduction}
Canards are trajectories of fast/slow dynamical systems that pass from an attracting slow manifold to a repelling slow manifold, and remain near the repelling slow manifold for $\mathcal{O}(1)$ time.  In smooth, planar systems with fast and slow nullclines that intersect transversely, the counterintuitive canard trajectories often appear as the limit cycles resulting from a Hopf bifurcation that occurs as the intersection point nears a fold (or local extremum) of the fast nullcline.  If the fast nullcline, also called the {\it critical manifold} is `S'-shaped, then the limit cycles grow into the relaxation oscillations that one would expect to see based on the stability of the slow manifolds.  The limit cycles, also called Hopf cycles or {\it canard cycles} are short lived, and the transition to the more intuitive relaxation oscillations happens so rapidly that the transition is called a {\it canard explosion}.  This paper extends canard explosion results for piecewise-smooth, continuous (PWSC) Li\'{e}nard systems.

Fast/slow systems are usually analyzed using {\it geometric singular perturbation theory} (GSPT).  GSPT relies on the ability to study the fast and slow dynamics separately, and provides the tools required to piece the full dynamics together from the parts.  However, the basic theory breaks down at local extrema of the critical manifold because the fast and slow dynamics become tangent (and hence there is no longer a separation of time scales, locally).   GSPT can be extended to degenerate points such as extrema using a various dynamical systems tools such as blow-up \cite{ksGSP}.  For an in-depth introduction to GSPT, the reader is directed to the paper by Jones \cite{gsp}.

If the slow nullclines do not intersect the critical manifold at a fold point, then the fold point is a singularity of the slow dynamics.  If the intersection occurs at the fold point, however, then the fold point becomes a removable singularity and is called a {\it canard point}.  In the singular limit (i.e., when the small parameter $\eps=0$), canard points allow trajectories of the slow flow to cross from from a stable branch of the critical manifold to an unstable branch (or vice versa).  When perturbing away from the singular limit, the trajectories that cross between branches with different stability perturb to canard trajectories for $\eps$ small enough \cite{ksGSP}.

In PWSC systems, a local extremum of the critical manifold may no longer be a fold (as in the case of Figure \ref{fig:2curve}), instead resulting from a discontinuous change in the sign of a derivative.  Such an extremum is called a {\it corner}.  When the slow nullcline intersects a corner, it does not result in a removable singularity of the slow flow.  In fact, for PWSC Van der Pol systems, $\eps$ can be {\it too small} for canards to exist \cite{pwlCanards, aar1}!

This paper will focus on PWSC Li\'{e}nard systems of the form
\begin{equation}
	\label{general}
	\begin{array}{l}
		\dot{x} = -y + F(x) \\
		\dot{y} = \eps g(x,y;\lambda, \eps)
	\end{array}
\end{equation}
where $$ F(x) = \left\{ \begin{array}{ll}
	f_-(x) & x\leq 0 \\
	f_+(x) & x \geq 0
	\end{array} \right. $$ 
with $f_-,f_+ \in C^k, k \geq 3$ such that $f_-(0)=f_+(0)=0$, $f_-'(0) \leq 0$, and $f_+'(0) \geq 0$ but omitting the case that $F \in C^1$ (i.e., $f_-'(0) = 0 = f_+'(0) $ is not allowed).  Additionally, it will be assumed that $f_+$ has a maximum at $x_M > 0$.  The restriction that either $f_-'(0) < 0$ or $f_+'(0) > 0$ (or both) is imposed, so that $F$ is not differentiable at $0$.  In this case, the critical manifold $$M_0 = \{y= F(x) \}$$ is `2'-shaped with a smooth fold at $x_M$ and a corner at $x=0$ as in Figure \ref{fig:2curve}.  Since $F'(0)$ does not exist, the set $\{x=0\}$ will be called the {\it splitting line}.

\begin{figure}[t]
                \includegraphics[width=0.45\textwidth]{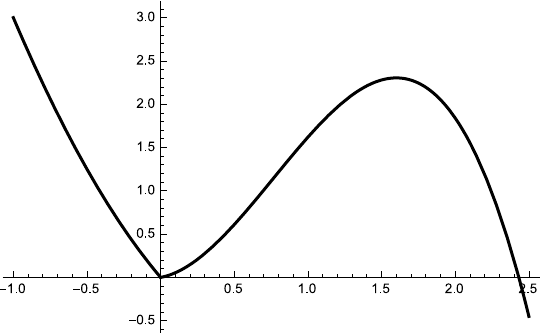}
                \caption{`2'-shaped critical manifold. }
                \label{fig:2curve}    
\end{figure}

Historically, a number of different techniques have been used to study canards in smooth, planar systems.  Benoit {\it et al.} first studied canard phenomena using nonstandard analysis \cite{benoitEA}.  Since the trajectories resemble ducks (see Figure \ref{fig:smoothDuck}), this group named them ``canards,'' and they are referred to as such (or as ``French ducks'') throughout the literature.  After Benoit {\it et al.}, Eckhaus used matched asymptotic expansions to examine canards \cite{eck}.  Recently, canards have analyzed through blow-up techniques.  The idea was first introduced by Dumortier and Roussarie \cite{dumort96} and later adopted by Krupa and Szmolyan \cite{ksGSP, ksRO}.  The use of blow-up techniques has been a significant development in the study of canards, allowing mathematicians to consider canard phenomena in higher dimensions.  In planar systems, canard trajectories only exist for an exponentially small parameter range, however they are more robust in higher dimensions \cite{brons, mmosurvey, szmolwechs, mw05}.

Canard research was reinvigorated by the implications in higher dimensions, eventually leading to a search for canards in piecewise-smooth systems.  The exploration began with a search into piecewise-linear systems by Rotstein, Coombes, and Gheorghe \cite{rotstein}.  Then Desroches {\it et al.} performed further analysis in the piecewise-linear Van der Pol case, demonstrating the existence of the {\it super-explosion} bifurcation, where a system transitions from a state with an attracting equilibrium instantaneously into a state with large relaxation oscillations, forgoing the canard explosion \cite{pwlCanards}.  Roberts and Glendinning extended these result to nonlinear piecewise-smooth Van der Pol systems, showing that nonlinearity plays a key role in the shape of the canard explosion.  The possibility of a subcritical super-explosion---a super-explosion bifurcation where an attracting relaxation oscillation appears prior to the destruction of the stable equilibrium---was also demonstrated \cite{aar1}.  In higher dimensions, Prohens and Teruel showed that a unique canard trajectory persists outside of the singular limit in 3D piecewise linear systems \cite{prohens13}.

\begin{figure}[t]
        \centering
        \begin{subfigure}[b]{0.45\textwidth}
                \includegraphics[width=\textwidth]{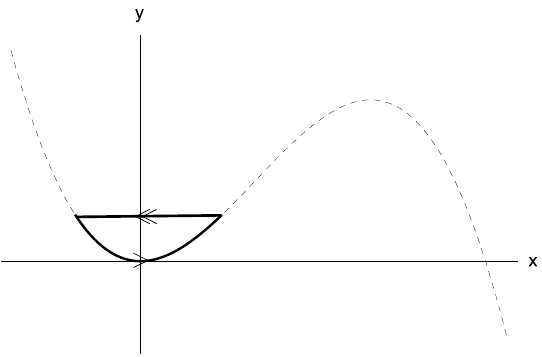}
                \caption{Singular canard ``without head.'' }
                \label{fig:smoothNoHead}
        \end{subfigure}
        ~ 
        \begin{subfigure}[b]{0.45\textwidth}
                \includegraphics[width=\textwidth]{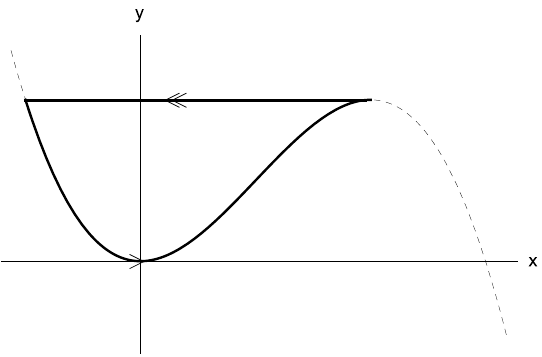}
                \caption{Singular ``maximal'' canard cycle.}
                \label{fig:smoothMax}
        \end{subfigure}
        \\ \vspace{0.4in}
                \centering
        \begin{subfigure}[b]{0.45\textwidth}
                \includegraphics[width=\textwidth]{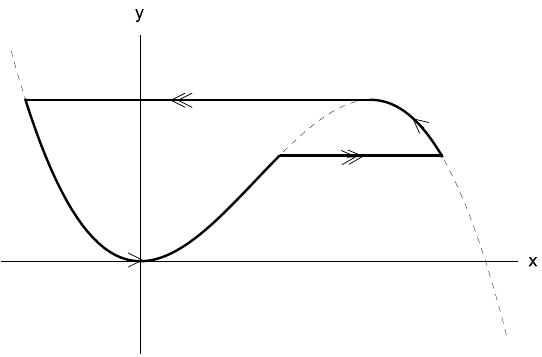}
                \caption{Singular canard ``with head.''}
                \label{fig:smoothHead}
        \end{subfigure}
        ~ 
        \begin{subfigure}[b]{0.45\textwidth}
                \includegraphics[width=\textwidth]{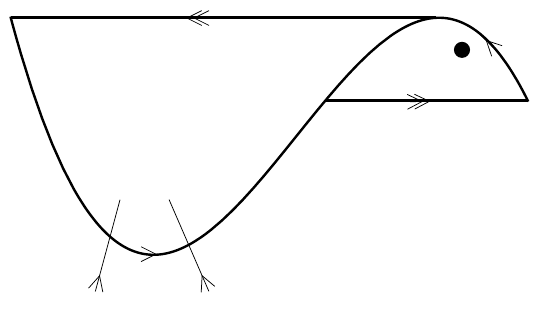}
                \caption{A duck!}
                \label{fig:smoothDuck}
        \end{subfigure}
        \caption{Examples of the three types of singular canard cycles in smooth systems: (a) canard without head, (b) maximal canard, and (c) canard with head.}
        \label{fig:smoothCanards}      
\end{figure}

This paper further extends the results of Roberts and Glendinning by allowing for more complicated slow dynamics.  Since the Hopf bifurcation near the smooth fold of the `2'-shaped critical manifold is still a smooth Hopf bifurcation, only the actual bifurcation point changes.  However, it is shown in Theorem \ref{thm:corner} that the more complicated slow dynamics affects type of bifurcation that occurs near the corner.  In the PWS Van der Pol case, the slow nullcline is a vertical line and there is a bifurcation creating stable periodic orbits as the nullcline passes through the corner of the `2'-shaped manifold \cite{aar1}.  In contrast to the smooth case, Theorem \ref{thm:corner} shows that the Hopf bifurcation may still occur at the critical manifold of the extremum even if the slope of the slow nullcline is not perfectly vertical.  That is, there is an open set in parameter space for which the bifurcation happens precisely at the corner.

The canard explosion phenomenon involves both local dynamics near a bifurcation point and global dynamics of a fast/slow system.  The existence or non-existence of stable canard orbits relies on the nature of the bifurcation near an extremum of the `2'-shaped manifold, and conditions determining the type of bifurcation are found.  The global analysis will be similar to that of \cite{aar1}, utilizing a {\it shadow system} that extends $f_+$ into the left half-plane (so that it is at least $C^3$).  The shadow system has the form:
\begin{equation}
	\label{shadow}
	\begin{array}{rl}
		\dot{x} &= -y + f_+(x) \\
		\dot{y} &= \eps g(x,y;\lambda,\eps).
	\end{array}
\end{equation}
Trajectories of \eqref{general} and \eqref{shadow} coincide in the right half-plane, but not for $x<0$.  The different behavior in the left half-plane allows the trajectories of both systems to be compared, showing that trajectories of the shadow system bound trajectories of the nonsmooth system. 

Section \ref{section:results} discusses the results.  First, the shadow system is discussed in \ref{sub:shadow}.  Next, a result about canard cycles near $x_M$ is presented in \ref{sub:smooth}.  Then the focus turns to canards near the origin in \ref{sub:corner}.  Finally, the paper concludes with a discussion in section \ref{section:discussion}.

\section{Results}
\label{section:results}
Given a system of the form \eqref{general}, the critical manifold is defined to be
$$M_0 = \{ y = F(x) \}.$$
The different branches of $M_0$ will be denoted
\begin{align*}
	M^l &= \{ (x, F(x) ): x<0 \} =  \{ (x, f_-(x) ): x<0 \} \\
	M^m &= \{ (x, F(x)) : 0 < x < x_M \} =  \{ (x, f_+(x)) : 0 < x < x_M \}  \\
	M^r &= \{ (x, F(x)): x > x_M \} = \{ (x, f_+(x)): x > x_M \}.
\end{align*}
$M^l$ and $M^r$ are attracting while $M^m$ is repelling.  

\subsection{The shadow system}
\label{sub:shadow}

The system \eqref{general} will be referred to as the true system.  Since $f_+'(0)$ exists, and in fact $f_+ \in C^3$ on the right half-plane, it can be extended to a $C^3$ function that is defined for $x < 0$.  Doing so produces the shadow system \eqref{shadow} whose trajectories agree with those of \eqref{general} for $x>0$.  The following lemma describes the relationship of the trajectories for $x<0$.

\begin{lemma}[Shadow Lemma]
\label{lem:shadow}
Let $\gamma_r(t) = (x_r(t), y_r(t))$ denote a trajectory of the true system \eqref{general}.  Assume $\gamma_r$ crosses the $y$-axis, entering the left half-plane at $\gamma_n(0) = (0, y_c)$.  Also consider the analogous trajectory $\gamma_s$ of the shadow system \eqref{shadow}, with $f_+(x) < f_-(x)$ for $x<0$.  Then, the radial distance from the origin of $\gamma_n$ is bounded by $\gamma_s$.
\end{lemma}

\begin{proof}
Let 
$$R(x,y) = \frac{x^2+y^2}{2}.$$
Then, the change in $R$ over time depends on the evolution of $x$ and $y$, and hence which vector field produces the trajectories when $x < 0$.  Denote $R_r$ to be the evolution of $R$ under the true system and $R_s$ to be the evolution of $R$ under the shadow system.  Then
\begin{align*}
	\dot{R}_r (x,y) &= x( f_-(x) - y ) + \eps y g(x,y,\eps) \\
	\dot{R}_s (x,y) &= x( f_+(x) - y ) + \eps  y g(x,y,\eps).
\end{align*}
for $x \leq 0$.  For a given $(x,y)$, 
$$ \dot{R}_r (x,y) - \dot{R}_s (x,y) = x [ f_-(x) - f_+(x) ] \leq 0,$$
since $x \leq 0$ and $f_-(x) < f_+(x)$, where equality holds only if $x = 0$.  Thus, the vector field of \eqref{general} points ``inward'' on trajectories of the shadow system.  Both $\gamma_r$ and $\gamma_s$ enter the left half-plane through the point $(0,y_c)$.  It remains to show that for $x<0$, $\gamma_r$ passes below $\gamma_s$.  That is, showing $R( \gamma_r (\delta t)) < R (\gamma_s (\delta t))$ for $\delta t > 0$ sufficiently small will complete the proof.  Since the vector fields \eqref{general} and \eqref{shadow} coincide on the $y$ axis, the comparison will be made using second order terms:
\begin{align*}
	\left( \begin{array}{c}  x_r (\delta t) \\ y_r (\delta t)  \end{array} \right)&=
		\left( \begin{array}{c} 
			0 + \dot{x}_r \delta t + \ddot{x}_r \delta t^2 + \ldots \\
			y_c + \dot{y}_r \delta t + \ddot{y}_r \delta t^2 + \ldots \end{array}  \right) \\
	\left( \begin{array}{c}  x_s (\delta t) \\ y_s (\delta t)  \end{array} \right)&=
		\left( \begin{array}{c} 
			0 + \dot{x}_s \delta t + \ddot{x}_s \delta t^2 + \ldots \\
			y_c + \dot{y}_s \delta t + \ddot{y}_s \delta t^2 + \ldots \end{array}  \right).
\end{align*}
The $0^{th}$ and $1^{st}$ order terms agree, so the important terms are
\begin{equation*}
	\begin{array}{lll}
		\ddot{x}_r =&	 - \dot{y}_r + f_-'(0) \dot{x}_r   	& = - \eps g(0,y_c; \lambda, \eps) - y_c f_-'(0) 		\\
		\ddot{x}_s =& 	-\dot{y}_s + f_+'(0) \dot{x}_s 	&= - \eps g(0,y_c; \lambda, \eps) - y_c f_+'(0) 			
	\end{array}
\end{equation*}		

and
\begin{equation*}
	\begin{array}{lll}
		\ddot{y}_r =& 	\eps (g_x \dot{x}_r + g_y \dot{y}_r )	& = \eps (- g_x y_c + g_y g(0,y_c; \lambda, \eps) )					\\
		\ddot{y}_s =&	\eps (g_x \dot{x}_s + g_y \dot{y}_s) 	& = \eps (- g_x y_c + g_y g(0,y_c; \lambda, \eps)) .
	\end{array}
\end{equation*}

Since it is assumed the vector fields point left at $(0,y_c)$, it is seen that $\dot{x} < 0$, and hence $y_c >0$.  The $\ddot{y}$ terms agree in both systems, so he relevant terms are the $\ddot{x}$ terms.  Because $-f_-'(0) > -f_+'(0)$, the trajectory of the shadow system moves further left for the same vertical change as the trajectory of the true system (see Figure \ref{fig:arrows}).  Thus, $\gamma_r$ enters the region $x<0$ below the shadow trajectory $\gamma_s$.  Since $y_c$, $\gamma_r$ will be closer to the origin than the shadow trajectory, and once inside is bounded by $\gamma_s$, proving the result.

 \begin{figure}[t]
	\begin{center}
		\includegraphics[width=0.6\textwidth]{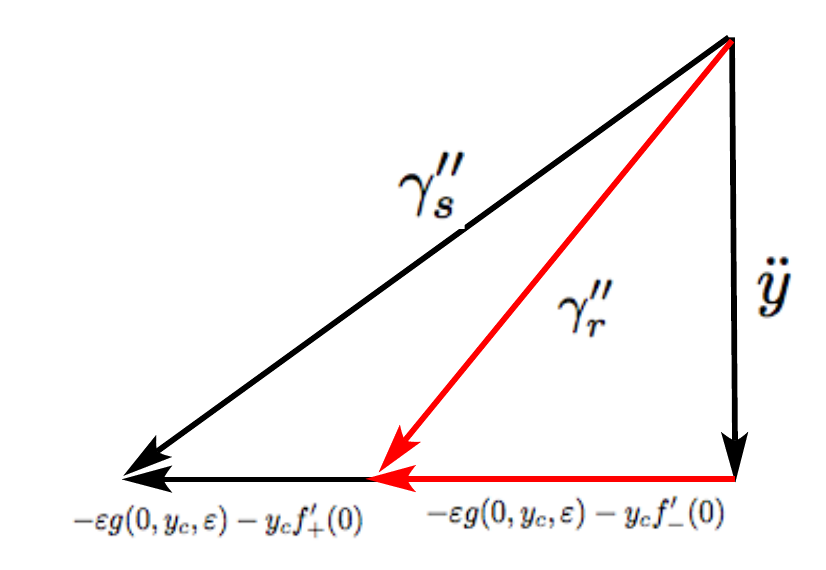}
			\caption{Relative slopes of vectors as discussed in the proof of Lemma \ref{lem:shadow}. }
			\label{fig:arrows}
	\end{center}
\end{figure}
\end{proof}

\begin{corollary}
\label{lem:CorShad}
Let $\gamma_r(t) = (x_r(t), y_r(t))$ denote a trajectory of the true system \eqref{general}.  Assume $\gamma_r$ crosses the $y$-axis, entering the left half-plane at $\gamma_n(0) = (0, y_c)$.  Also consider the analogous trajectory $\gamma_s$ of the system
	\begin{equation}
		\begin{array}{rl}
			\dot{x} &= -y + \tilde{F}(x) \\
			\dot{y} &= \eps g(x,y; \lambda, \eps),
		\end{array}
	\end{equation}
	where $$ \tilde{F}(x) = \left\{ \begin{array}{ll}
	\tilde{f}_-(x) & x\leq 0 \\
	f_+(x) & x \geq 0
	\end{array} \right.$$ 
and $\tilde{f}_-'(0) > f_-'(0)$. Then, the radial distance from the origin of $\gamma_n$ is bounded by that of $\gamma_s$.
\end{corollary}

\begin{proof}
The proof is identical to that of Lemma \ref{lem:shadow}
\end{proof}

\subsection{Canard cycles near the smooth fold}
\label{sub:smooth}
Previous work on canards in PWSC systems has focused on Van der Pol type systems---that is, systems where $g(x,y;\lambda,\eps) = x- \lambda$ \cite{pwlCanards, aar1}.  In these systems with vertical slow nullclines, the Hopf (or Hopf-like) bifurcations occur precisely when the slow nullcline passes through an extremum of the critical manifold.  In more general systems such as \eqref{general}, the bifurcation point may move.  The nature of the growth of canard cycles requires an interplay between the local analysis near a Hopf bifurcation and global dynamics involving some form of return mechanism (often through an `S'-shaped critical manifold).  This interplay is the reason the explosion in the PWL case \cite{pwlCanards} is different from that of the nonlinear PWS case \cite{aar1}.  The shadow lemma above generalizes a lemma from \cite{aar1}, essentially guaranteeing that the global analysis will be unchanged by the more general slow dynamics.  The bulk of the work in the remainder of this section will be dedicated to local analysis near a Hopf bifurcation or some nonsmooth variant.

The first case that will be considered is when the Hopf bifurcation occurs near the smooth fold at $x_M$.  The local dynamics here are that of the smooth system \eqref{shadow} since $x_M >0$.  Classical results discuss the nature of the Hopf bifurcation (i.e., local behavior), but technically do not discuss the full nature of the canard explosion since $F,$ and consequently the vector field, is not $C^3$.

\begin{theorem}
	\label{thm:fold}
	Let $0 < \epsilon \ll 1$ be sufficiently small and suppose the parameter $\lambda$ indicates the $x$-coordinate of a unique and transverse intersection point of the sets $\{ g(x,y;\lambda, \eps) =0 \}$ and $\{ y = F(x) \}$.  If  $$g_x  (x_M , F(x_M) ; x_M , \eps ) > 0, $$ then the system \eqref{general} undergoes a Hopf bifurcation for some $\lambda_M$ in a neighborhood of $x_M$.
		\begin{enumerate}[(i)]
			\item If the Hopf bifurcation is supercritical, then it will produce stable canard cycles.  
			\item If the Hopf bifurcation is subcritical, then it will produce stable relaxation oscillations.   Additionally, unstable canard orbits will exist for values of $\lambda$ in an $\mathcal{O}(\eps)$ range prior to the bifurcation, and the canard orbits will be destroyed through the bifurcation.
		\end{enumerate}
\end{theorem}

\begin{figure}[t]
	\begin{center}
		\includegraphics[width=0.6\textwidth]{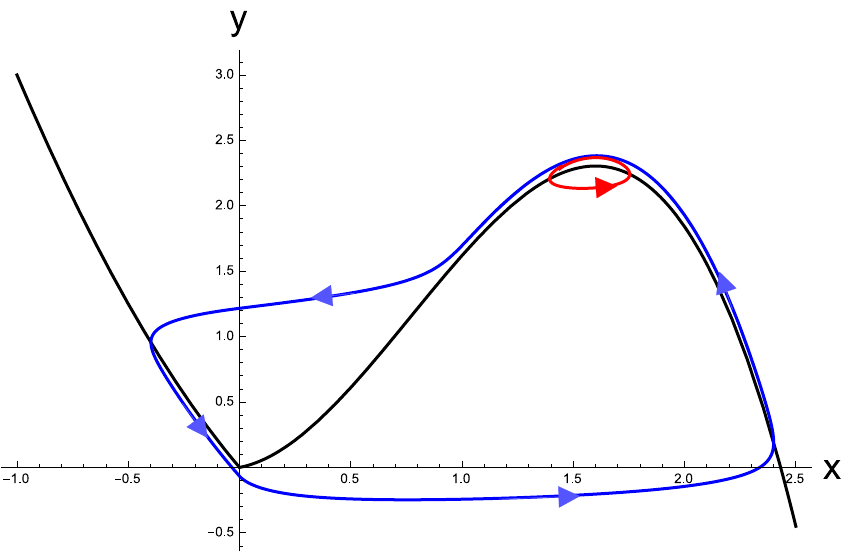}
			\caption{Example of a canard cycle (red) and canard with head (blue) at the smooth fold when $\eps = 0.1.$  Trajectories computed using NDSolve in Mathematica.}
			\label{fig:fold}
	\end{center}
\end{figure}\

\begin{proof}
The point $(\lambda, F(\lambda))$ is an equilibrium since it occurs at an equilibrium of the nullclines.  Linearizing about this equilibrium provides the Jacobian
$$ \mathcal{J} = \left( \begin{array}{cc}
	F'(\lambda) & -1 \\
	\eps g_x (\lambda, F(\lambda); \lambda, \eps) & \eps g_y (\lambda, F(\lambda); \lambda, \eps)
	\end{array} \right).$$
	Since $g_x (x_M, F(x_M); x_M, \eps) >0$, there exists a neighborhood $V(x_M)$ such that for all $\lambda \in V(x_M)$, $g_x (\lambda, F(\lambda); \lambda, \eps) > 0$.  Also, since $F'(x_M) = f_+'(x_M) = 0$, there exists a $\lambda_H$ in a neighborhood of $x_M$ such that 
	$$F'(\lambda_H) = - \eps g_y (\lambda_H, F(\lambda_H); \lambda_H, \eps).$$  
	If $\eps$ is sufficiently small, it is clear that $\lambda_H \in V(x_M)$.  Thus, there is a Hopf bifurcation at $(\lambda_H, F(\lambda_H))$.  Since the bifurcation is smooth, its criticality can be determined by the formulae found in \cite{glendinning1994stability} or \cite{ksRO}, for example.  
	The proof of the existence canard cycles or relaxation oscillations is identical to that of Theorem II.3 from \cite{aar1}.
	\end{proof}

Combining Theorem \ref{thm:fold} and Lemma \ref{lem:shadow} provides the following corollary bounding the size of the periodic orbits created by the Hopf bifurcation.  Figure \ref{fig:fold} depicts the result.

\begin{corollary}
	\label{cor:fold}
	Let $\Gamma_s$ be periodic orbits created through a Hopf bifurcation near $x_M$ in the shadow system \eqref{shadow} and let $\Gamma_n$ be the periodic orbits created by the Hopf bifurcation in the PWS system \eqref{general}.  Then $\Gamma_n$ are bounded by $\Gamma_s$.  
\end{corollary}

 \begin{figure}[t]
	\begin{center}
		\includegraphics[width=0.6\textwidth]{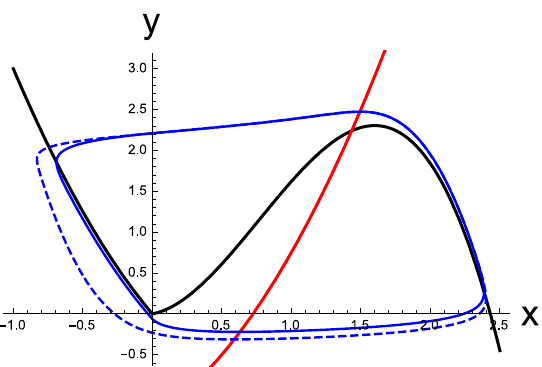}
			\caption{A `2'-shaped critical manifold is intersected by a nonlinear slow nullcline near the smooth fold at $x_M$.  The stable periodic orbit $\Gamma_s$ for shadow system \eqref{shadow} (dashed trajectory) bounds the stable periodic orbit $\Gamma_n$ of the true system \eqref{general} (solid trajectory).  Trajectories computed using NDSolve in Mathematica.}
			\label{fig:shadow}
	\end{center}
\end{figure}

\subsection{Canard cycles at the origin}
\label{sub:corner}
It should not be surprising that it is possible to have canard cycles near $x_M$.  In light of \cite{pwlCanards} and \cite{aar1}, one should expect to see canards near the corner at the origin as well, provided certain conditions are met.  By incorporating more complicated slow dynamics, the bifurcation point may be moved away from the corner.  However, there may still be a nonsmooth bifurcation at the corner, and it is not quite the special case that it would be in a smooth system.  The following theorem describes the periodic orbits that are created through a bifurcation at or near the corner at the origin.  
\begin{theorem}
\label{thm:corner}
	Given a system of the form \eqref{general} with $f_-,f_+ \in C^k, k \geq 3$, $f_-(0)=f_+(0)=0$, $f_-'(0) < 0$, and $f_+'(0) > 0$, suppose that $g_x(0,0;0,\eps) > \max \{- f_+'(0) g_y(0,0;0,\eps), -f_-'(0) g_y(0,0;0,\eps) \}.$  Then, for $\eps$ sufficiently small, there exists a $\lambda_0$ in a neighborhood of $0$ such that an attracting periodic orbit $\Gamma_n(\lambda)$ is created through a bifurcation when $\lambda = \lambda_0$.  Let $\alpha_\pm = f_\pm'(0) + \eps g_y(0,0;\lambda_0,\eps)$ and $\beta_\pm = [\eps g_y(0,0;\lambda_0,\eps) - f_\pm '(0)]^2 - 4 \eps g_x (0,0;\lambda_0,\eps)$.  The bifurcation is described by the following:
	\begin{enumerate}[(i)]
		\item If $ \alpha_- > 0$, then the bifurcation is a smooth Hopf bifurcation and $\lambda_0 < 0$.  
		\item If $ \alpha_+ < 0$, then the bifurcation is a smooth Hopf bifurcation and $\lambda_0 >0$.  
		\item If $f_-'(0) \leq - \eps g_y (0,0; \lambda_0,\eps) \leq f_+'(0)$ (i.e, $\pm \alpha_\pm \geq 0$), then the bifurcation is nonsmooth and $\lambda_0=0$. 
			\subitem (a) If $\beta_+ < 0$ and $\beta_- < 0$, then there is a nonsmooth Hopf bifurcation.  Let 
				\begin{equation}
					\label{nsHCrit}
					\Lambda = \frac{\alpha_+}{\sqrt{-\beta_+}} - \frac{- \alpha_-}{\sqrt{- \beta_-}}.
				 \end{equation}
				 Then the bifurcation is supercritical if $\Lambda <0$ and subcritical if $\Lambda >0$.
			\subitem (b) If $\beta_+ < 0$ and $\beta_- \geq 0$, then the bifurcation is a supercritical Hopf-like bifurcation that creates small amplitude periodic orbits.
			\subitem(c) If $\beta_+ \geq 0$, then the bifurcation is a super-explosion.  The bifurcation is subcritical if $\beta_- < 0$ and is supercritical if $\beta_- \geq 0$.
	\end{enumerate}
	Figure \ref{fig:cornerCanards} depicts the orbits in the supercritical cases.
\end{theorem}

\begin{figure}[t]
        \centering
        \begin{subfigure}[t]{0.45\textwidth}
                \includegraphics[width=\textwidth]{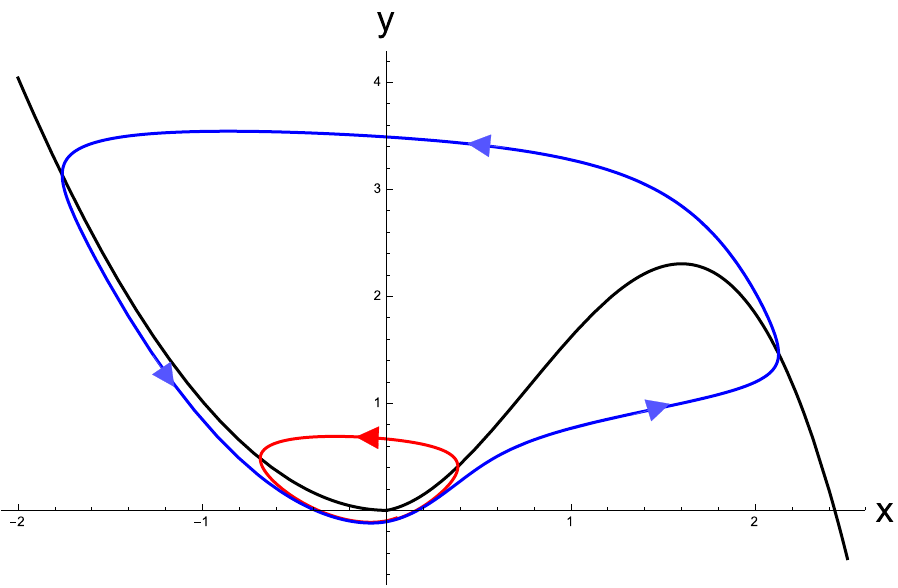}
                \caption{Canards in the case of Theorem \ref{thm:corner} (i) when $\eps = 0.1$.  Canard without head when $\lambda = -0.115$ (in red) and canard with head when $\lambda = -0.114$ (in blue).}
                \label{fig:alphaMP}
        \end{subfigure}
        ~ 
        \begin{subfigure}[t]{0.45\textwidth}
                \includegraphics[width=\textwidth]{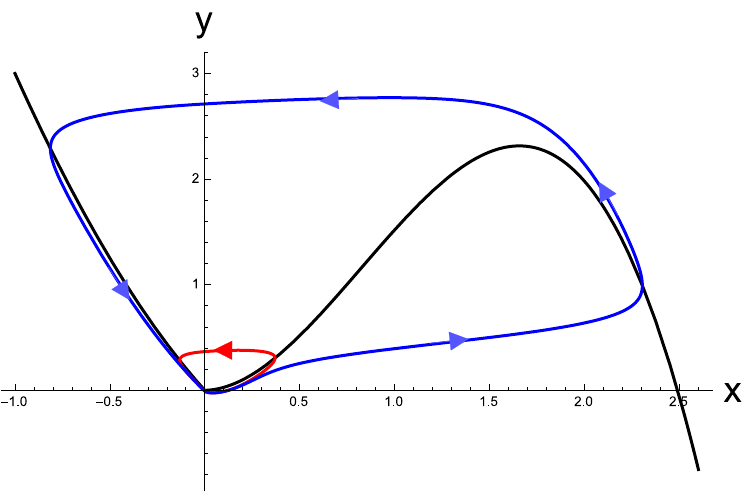}
                \caption{Canards in the case of Theorem \ref{thm:corner} (ii) when $\eps = 0.1$. Canard without head when $\lambda = 0.049$ (in red) and canard with head when $\lambda = 0.05$ (in blue).}
                \label{fig:alphaPN}
        \end{subfigure}
        \\ \vspace{0.4in}
                \centering
        \begin{subfigure}[t]{0.45\textwidth}
                \includegraphics[width=\textwidth]{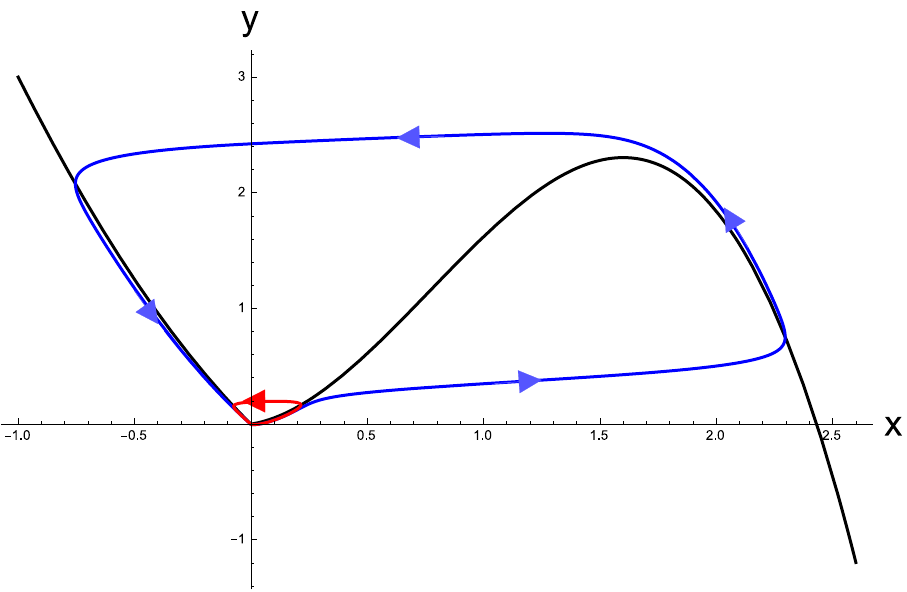}
                \caption{Canards in the case of Theorem \ref{thm:corner} (iii-b) when $\eps = 0.1$.  Canard without head when $\lambda = 0.01293$ (in red) and canard with head when $\lambda=0.01295$ (in blue). }
                \label{fig:alpha0C}
        \end{subfigure}
        ~ 
        \begin{subfigure}[t]{0.45\textwidth}
                \includegraphics[width=\textwidth]{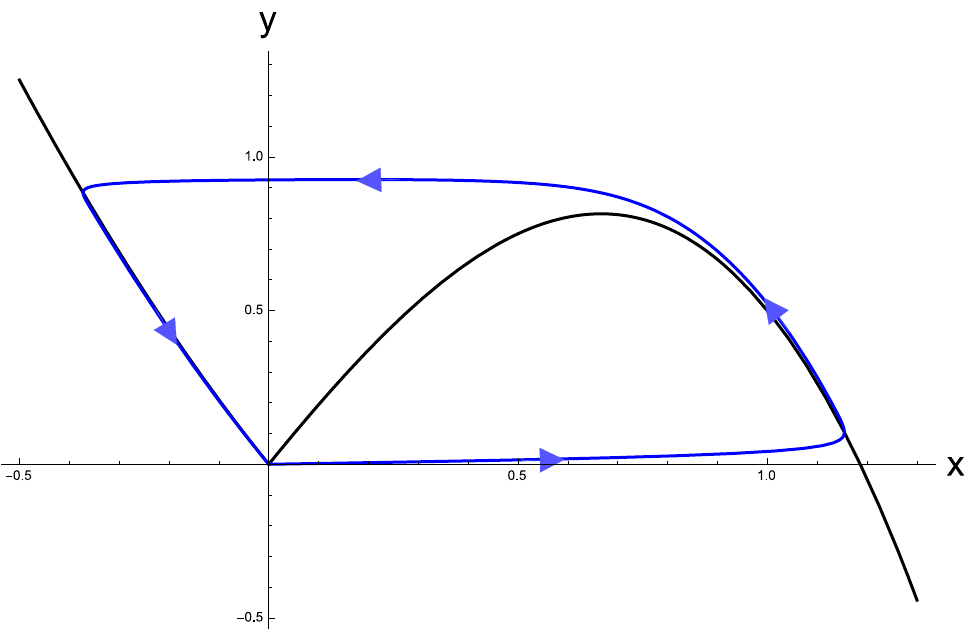}
                \caption{Super-explosion as in the case of Theorem \ref{thm:corner} (iii-c), for $\eps = 0.01$ and $\lambda = 10^{-7}$.}
                \label{fig:alpha0SE}
        \end{subfigure}
        \caption{Periodic orbits created through the bifurcation described in Theorem \ref{thm:corner}.  Trajectories computed using NDSolve in Mathematica.}
        \label{fig:cornerCanards}      
\end{figure}

\begin{proof}
By assumption, there is an equilibrium point at $(\lambda, F(\lambda))$, and for $\lambda \neq 0$, the linearization near the equilibrium is given by 
\begin{equation}
\label{jac}
\mathcal{J} = \left( \begin{array}{cc}
F'(\lambda) & -1 \\
	\eps g_x (\lambda, F(\lambda); \lambda, \eps) & \eps g_y (\lambda, F(\lambda); \lambda, \eps)
	\end{array} \right).
\end{equation}
Therefore, if $\pm \lambda > 0$, \eqref{jac} becomes
\begin{equation}
\label{jacPM}
\mathcal{J} = \mathcal{J_\pm} \left( \begin{array}{cc}
f_\pm'(\lambda) & -1 \\
	\eps g_x (\lambda, F(\lambda); \lambda, \eps) & \eps g_y (\lambda, F(\lambda); \lambda, \eps)
	\end{array} \right).
\end{equation}
The eigenvalues of $\mathcal{J_\pm}$ limit to 
$$ \alpha_\pm + \sqrt{\beta_\pm} \text{ and } \alpha_\pm - \sqrt{\beta_\pm}$$ as $\lambda \rightarrow 0^\pm.$  Also, since $f_+'(0) > 0 > f_-'(0)$, it is clear that $\alpha_+ > \alpha_-$.  

If $\alpha_- > 0$, for $\eps$ sufficiently small, there exists a $\lambda_*(\eps) <0$ such that $\text{Tr}(\mathcal{J}) = f_-'(\lambda_*) + \eps g_y(0,0;\lambda_*,\eps) < 0$.  Therefore, there exists a $\lambda_0 (\eps) \in (\lambda_*(\eps), 0)$ such that $\text{Tr}(\mathcal{J}) = 0$.  The condition that $g_x(0,0;0,\eps) > \max \{- f_+'(0) g_y(0,0;0,\eps), -f_-'(0) g_y(0,0;0,\eps) \}$ indicates that $\det (\mathcal{J})>0$, and there fore there is a smooth Hopf bifurcation at $\lambda_0(\eps) < 0 $.  A smilier argument can be used to show that there exists a $\lambda_0 (\eps) > 0$ such that a smooth Hopf bifurcation at $(\lambda_0,f_+(\lambda_0))$ if $\alpha_+ < 0$.  This proves parts (i) and (ii) of the theorem.

\begin{figure}[t]
        \centering
        \begin{subfigure}[t]{0.3\textwidth}
                \includegraphics[width=\textwidth]{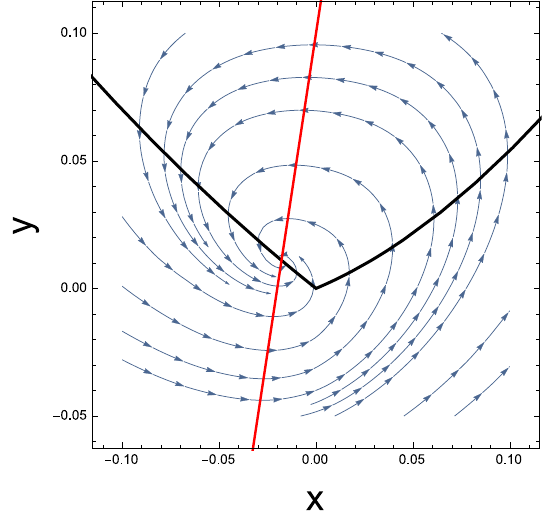}
                \caption{Phase portrait before bifurcation with $\lambda = -0.02$.  }
                \label{fig:canFBif1}
        \end{subfigure}
        ~ 
        \begin{subfigure}[t]{0.3\textwidth}
                \includegraphics[width=\textwidth]{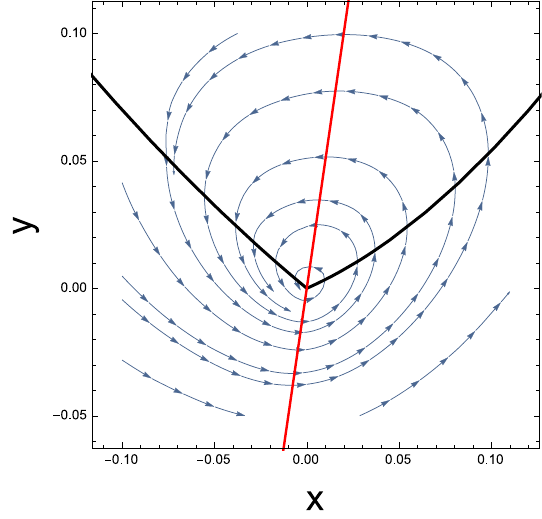}
                \caption{Phase portrait at bifurcation ($\lambda=0$).}
                \label{fig:canFBif2}
        \end{subfigure}
                ~ 
        \begin{subfigure}[t]{0.3\textwidth}
                \includegraphics[width=\textwidth]{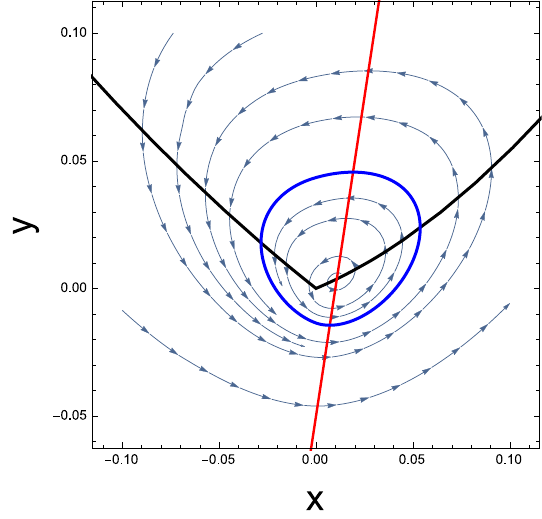}
                \caption{Phase portrait after bifurcation with $\lambda=0.01$.  The solid orbit (blue) depicts the attracting periodic orbit.}
                \label{fig:canFBif3}
        \end{subfigure}
        \caption{The nonsmooth Hopf bifurcation described in Theorem \ref{thm:corner} (iii-a) when $\eps = 0.1$ with critical manifold (black) and slow nullcline (red).  Trajectories computed using NDSolve in Mathematica.}
        \label{fig:canFBif}      
\end{figure}

If $\alpha_- \leq 0$ and $\alpha_+ \leq 0$, then $\text{Tr}(\mathcal{J}) \neq 0$ anywhere $\mathcal{J}$ is defined.  However, the equilibrium point at $(\lambda,F(\lambda))$ still transitions from having two stable eigenvalues to two unstable eigenvalues as $\lambda$ increases through zero.  If $\beta_+ < 0$, then the equilibrium will transition to being an unstable focus for $\lambda >0$ small enough.  If $\beta_- <0$ as well, then the transition is from a stable node to an unstable node.  Thus it is the nonsmooth equivalent of a Hopf bifurcation, discussed by Simpson and Meiss in \cite{simpmeiss}, and the criticality is determined by $\Lambda$.  Thus part (iii-a) is proved.  The bifurcation is depicted in Figure \ref{fig:canFBif}.

Now suppose $\beta_- \geq 0$.  Then the equilibria is a stable node for $\lambda <0$ and transitions to an unstable focus as $\lambda$ increases through zero.  Comparison with a shadow system will show that the bifurcation is supercritical.  Let the shadow system be 	\begin{equation}
\label{shadowPf}
		\begin{array}{rl}
			\dot{x} &= -y + \tilde{F}(x) \\
			\dot{y} &= \eps g(x,y; \lambda, \eps),
		\end{array}
	\end{equation}
	where $$ \tilde{F}(x) = \left\{ \begin{array}{ll}
	\tilde{f}_-(x) & x\leq 0 \\
	f_+(x) & x \geq 0
	\end{array} \right. $$ 
	with $\tilde{f}_-'(x) < 0$ for all $x < 0$ and $\tilde{f}_-'(0) =\min{0, -\eps g_y}.$  Define $\tilde{\alpha}_-,\tilde{\beta}_-$ to be the corresponding quantities for the shadow system.  Then $\tilde{\alpha}_- \leq 0$, and
	\begin{align*}
		\tilde{\beta}_- & = [ \tilde{f}_-(0) - \eps g_y(0,0;0,\eps)]^2 - 4 \eps g_x(0,0;0,\eps) \\
					& \leq [2 \eps g_y (0,0;0,\eps)]^2 - 4 \eps g_x(0,0;0,\eps)\\
					&= 4 \eps [ \eps g_y(0,0;0,\eps)^2 - g_x(0,0;0,\eps)]
					& <0
	\end{align*}
	since $g_x(0,0;0,\eps) > \max \{- f_+'(0) g_y(0,0;0,\eps), -f_-'(0) g_y(0,0;0,\eps) \}$.  Therefore, the shadow system\eqref{shadowPf} satisfies the conditions of (iii-b), and for $\lambda >0$ small enough there will be a stable periodic orbit.  By Lemma \ref{lem:CorShad}, orbits of the main system \eqref{general} are bounded inside the orbits of the shadow system.  Since the only equilibrium point inside the periodic orbit
is repelling, the Poincar\'{e}-Bendixson theorem guarantees the existence of an attracting periodic orbit, and this orbit will be bounded inside the periodic orbit of the shadow system.  This proves (iii-b).  The bifurcation is depicted in Figure \ref{fig:canBif}.

\begin{figure}[t]
        \centering
        \begin{subfigure}[t]{0.3\textwidth}
                \includegraphics[width=\textwidth]{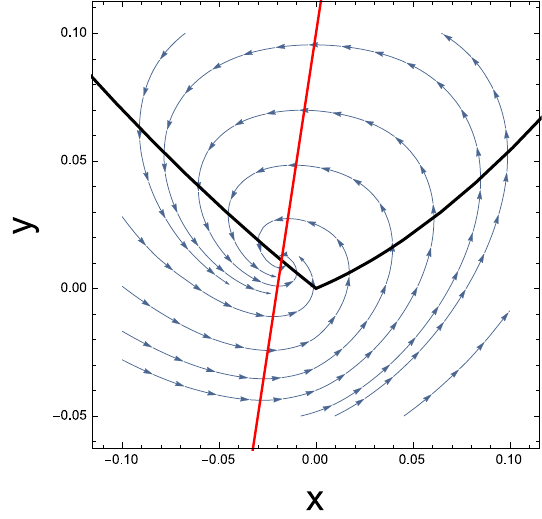}
                 \caption{Phase portrait before bifurcation with $\lambda = -0.05$.  }
                \label{fig:canBif1}
        \end{subfigure}
        ~ 
        \begin{subfigure}[t]{0.3\textwidth}
                \includegraphics[width=\textwidth]{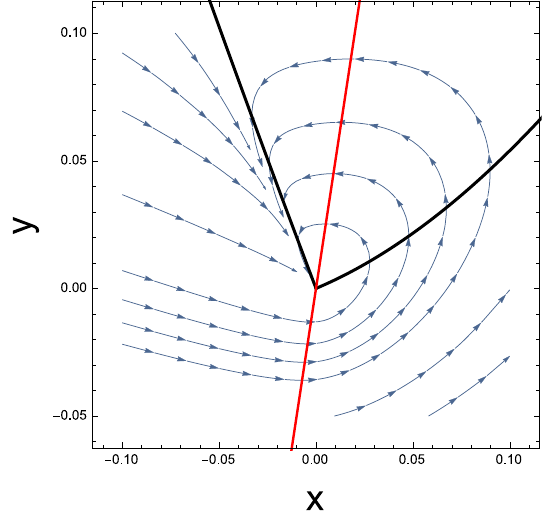}
                \caption{Phase portrait at bifurcation ($\lambda=0$).}
                \label{fig:canBif2}
        \end{subfigure}
                ~ 
        \begin{subfigure}[t]{0.3\textwidth}
                \includegraphics[width=\textwidth]{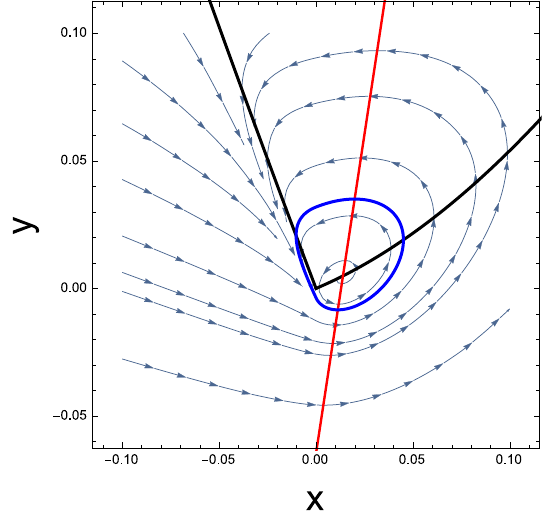}
                \caption{Phase portrait after bifurcation with $\lambda=0.013$. The solid orbit (blue) depicts the attracting periodic orbit.}
                \label{fig:canBif3}
        \end{subfigure}
       \caption{The nonsmooth Hopf-like bifurcation described in Theorem \ref{thm:corner} (iii-b) when $\eps = 0.1$ with critical manifold (black) and slow nullcline (red).  Trajectories computed using NDSolve in Mathematica.}
        \label{fig:canBif}      
\end{figure}

Finally, suppose $\beta_+ \geq 0$.  Then the equilibrium point either transitions from a stable node to unstable node (if $\beta_- \geq 0$) or from a stable focus to an unstable node (if $\beta_- < 0$).  First, consider the case $\beta_- \geq 0$.  For $\lambda$ small, but positive, the Jacobian of the system at the equilibrium point is 
\begin{equation}
\label{jacSE}
\mathcal{J} = \left( \begin{array}{cc}
f_+'(\lambda) & -1 \\
	\eps g_x (\lambda, f_+(\lambda); \lambda, \eps) & \eps g_y (\lambda, f_+(\lambda); \lambda, \eps)
	\end{array} \right).
\end{equation}
Let $v_1 = (a_1,b_1),\ v_2=(a_2,b_2)$ be the eigenvectors of $\mathcal{J}$ corresponding to $0 < \mu_1 < \mu_2$, respectively.  
Then
\begin{align*}
	f_+'(\lambda) a_i - b_i &= \mu_i a_i \\
	\eps g_x a_i + \eps g_y b_i &= \mu_i b_i,
\end{align*}
and manipulating the second equations gives
$$ \frac{b_i}{a_i} = \frac{\eps g_x}{\mu_i - \eps g_y}.$$
Since $\mu_2 \rightarrow 2 f_+'(\lambda) \neq 0$ as $\eps \rightarrow 0$, the slope of $v_2$ can be made arbitrarily small (and positive) for $\eps$ small enough.  Therefore, the strong unstable trajectory leaves the equilibrium point and crosses the critical manifold to the right of $x = x_M$.  From there, the trajectory tracks the critical manifold passing above the point $(x_M, f_+(x_M))$ on its way to the left half-plane.  Once in the left half-plane, the trajectory proceeds left until it crosses the critical manifold again, and then moves down and to the right until it crosses the $x$-axis again, this time below the origin, and then eventually it intersects the slow nullcline.  Let $V$ be the set enclosed by the trajectory and slow nullcline, as shown in Figure \ref{fig:V}.  Then all trajectories outside of $V$ are bounded outside of $V$.  Additionally, it is possible to construct a positively invariant set $W$, similar to the one described in \cite{aar1}. $W$ is also depicted in Figure \ref{fig:V}.  The set $W \setminus V$ is positively invariant and contains no stable equilibria, so the existence of an attracting periodic orbit is guaranteed by the Poincar\`{e}-Bendixson theorem.  Since this orbit is bounded away from the repelling branch of the critical manifold, it must be a relaxation oscillation.  Therefore, as $\lambda$ increases through zero, the system \eqref{general} undergoes a super-explosion bifurcation. 

Figure \ref{fig:expBif} depicts the supercritical super-explosion bifurcation.  The transition from stable equilibrium to stable relaxation oscillation through a supercritical bifurcation is fundamentally different from what is observed in the smooth case, and thus is a consequence of the piecewise nature of the vector field.  One reason for the difference is that the equilibrium at the bifurcation point is globally attracting, but not locally stable.  In the vector field defined on the left half-plane, the equilibrium at the origin is a stable node, while it is an unstable node in the vector field on the right half-plane.  Figure \ref{fig:expBif2} shows the equilibrium with its strong stable manifold in the left half-plane and its global strong unstable manifold in the right half-plane.  The strong unstable manifold eventually enters the left half-plane and returns to the equilibrium point, creating a homoclinic orbit.  This homocolinc orbit forms a separatrix between the two basic types of orbits in this system.  The region inside this trajectory consists of a family of homoclinic orbits.  All trajectories outside this region only approach the equilibrium in forward time.  Again looking at Figure \ref{fig:expBif2}, it is clear that the only trajectories originating in the left half-plane that cross over into the right half-plane do so below the strong stable manifold and pass below the strong unstable manifold (before eventually returning to the left half-plane).  This shows that no trajectories pass from a stable slow manifold to an unstable slow manifold.  Figure \ref{fig:expBif3} shows that this is the case after the bifurcation as well.

\begin{figure}[t]
        \centering
        \begin{subfigure}[t]{0.3\textwidth}
                \includegraphics[width=\textwidth]{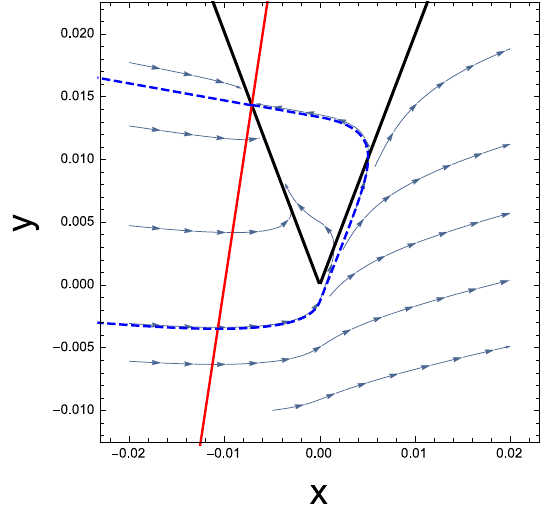}
                 \caption{Phase portrait before bifurcation with $\lambda = -0.01$.  Dashed blue trajectory depicts the strong stable trajectory to the folded node.}
                \label{fig:expBif1}
        \end{subfigure}
        ~ 
        \begin{subfigure}[t]{0.3\textwidth}
                \includegraphics[width=\textwidth]{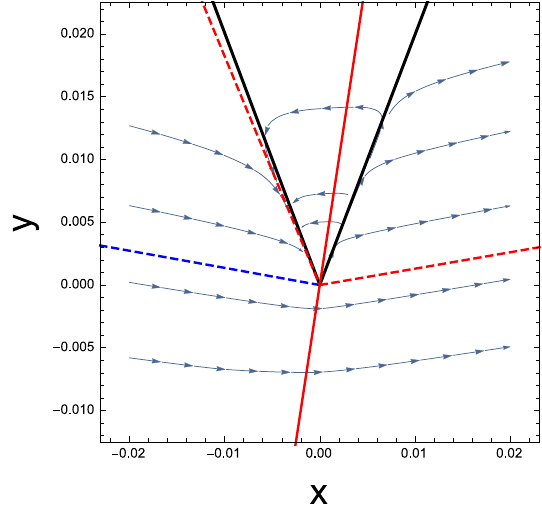}
                \caption{Phase portrait at bifurcation ($\lambda=0$).  Dashed trajectories depict strong stable (blue) and unstable (red) manifolds.}  Notice that the global strong unstable trajectory forms a homoclinic orbit.
                \label{fig:expBif2}
        \end{subfigure}
                ~ 
        \begin{subfigure}[t]{0.3\textwidth}
                \includegraphics[width=\textwidth]{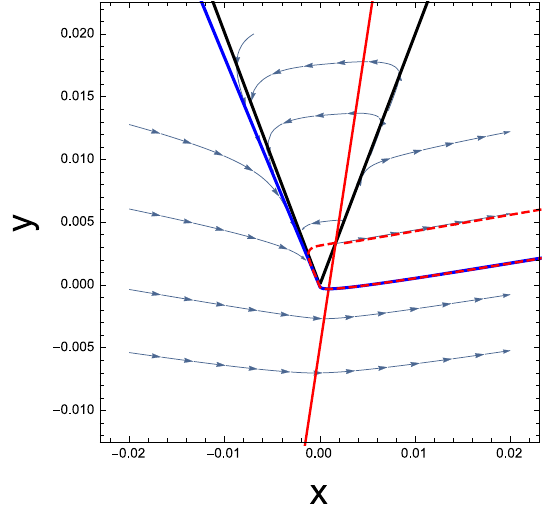}
                \caption{Phase portrait after bifurcation with $\lambda=0.013$.  Dashed red trajectory depicts strong unstable trajectory.  The solid orbit (blue) depicts a portion of the stable periodic orbit created through the bifurcation.}
                \label{fig:expBif3}
        \end{subfigure}
       \caption{The supercritical super-explosion bifurcation described in Theorem \ref{thm:corner} (iii-c) when $\eps = 0.01$ with critical manifold (black and slow nullcline (red).  Trajectories computed using NDSolve in Mathematica.}
        \label{fig:expBif}      
\end{figure}

Now, let $\beta_- < 0$ so that the equilibrium point at $(\lambda,f_-(\lambda))$ is a stable focus for $\lambda < 0$.  Then for $|\lambda|$ sufficiently small, there exists a $z_+$ such that $0 < z_+ < f_+(x_M)$ the trajectory through $(0,z+)$ spirals around the equilibrium in the left half plane and intersects the $x$-axis again at a point $(0,z_-)$ with $z_- < 0$.  The dynamics in the right half-plane are governed by an unstable node in the shadow system given by \eqref{shadow}, sending the trajectory to the right across the critical manifold, up over the point $(x_M,f_+(x_M))$ and back to the $x$-axis at a point $(0,z')$ where $z' > z_+$.  Let $V'$ denote the set bounded by this trajectory and the segment of the $x$-axis connecting $z_+$ to $z'$, as shown in Figure \ref{fig:VP}.  Then $V'$ is negatively invariant and contains only a stable equilibrium.  Therefore, there is a repelling periodic orbit inside $V'$.  Additionally, the set $W \setminus V$ is positively invariant and contains no equilibria.  So there is an attracting periodic orbit that is bounded away from the repelling branch of the critical manifold.  Thus an attracting equilibrium, repelling periodic orbit, and attracting relaxation oscillation exist simultaneously for $\lambda <0$.  Therefore the bifurcation that occurs as $\lambda$ increases through zero is a subcritical super-explosion.

\begin{figure}[t]
        \centering
        \begin{subfigure}[t]{0.45\textwidth}
                \includegraphics[width=\textwidth]{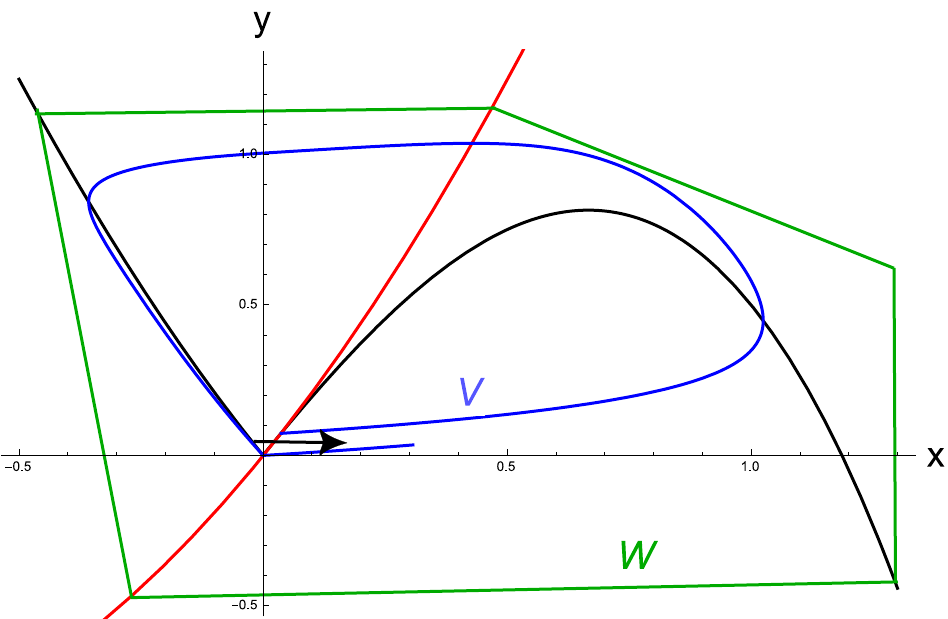}
                \caption{The sets $V$ and $W$ from the proof of Theorem \ref{thm:corner}, part (iii-c) when $\lambda = 0.037$.}
                \label{fig:V}
        \end{subfigure}
        ~ 
        \begin{subfigure}[t]{0.45\textwidth}
                \includegraphics[width=\textwidth]{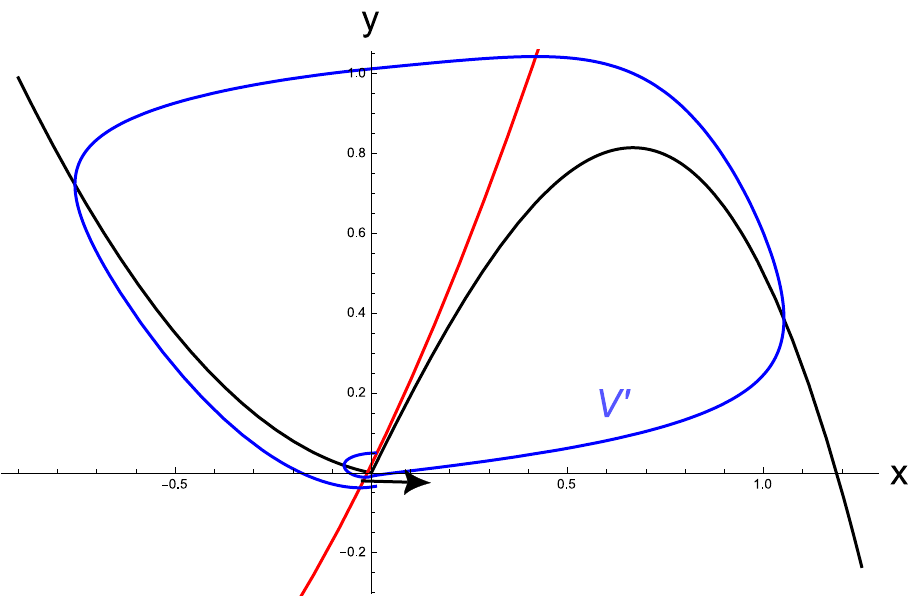}
                \caption{The set $V'$ from the proof of Theorem \ref{thm:corner}, part (iii-c) when $\lambda = -0.01$.  $W$ can be constructed similarly to Figure \ref{fig:V}. }
                \label{fig:VP}
        \end{subfigure}
        \caption{Sets showing the existence of a super-explosion bifurcation when $\eps = 0.1$.  (a) The bifurcation is supercritical.  The equilibrium point is repelling, and a stable periodic orbit exists in the region $W \setminus V$.  (b) The bifurcation is subcritical.  The equilibrium point is attracting, an unstable periodic orbit exits in $V'$, and a stable periodic orbit exists in $W \setminus V'$.  Trajectories computed using NDSolve in Mathematica.}
        \label{fig:vs}      
\end{figure}

\end{proof}
 
Applying the shadow lemma (Lemma \ref{lem:shadow}) provides an immediate corollary regarding the amplitudes of the periodic orbits created through the bifurcation described in Theorem \ref{thm:corner}.

\begin{corollary}
Fix $\eps, \lambda > 0$.  Let $\Gamma_n$ be an attracting periodic orbit in system \eqref{general} created through the bifurcation described in Theorem \ref{thm:corner}, and assume the shadow system \eqref{shadow} has an attracting periodic orbit $\Gamma_s$.  Then $\Gamma_n$ is bounded by $\Gamma_s$.  
\end{corollary}

\section{Discussion}
\label{section:discussion}
It has been shown that systems of the form \eqref{general} undergo bifurcations creating attracting periodic orbits as the slow nullcline passes near either a smooth fold at $x = x_M$ or a corner at the origin.  In \cite{aar1}, only piecewise-smooth Van der Pol systems are discussed, so the location of the bifurcation was known {\it a priori}.  In Van der Pol systems, the bifurcation always occurs at an extremum of the critical manifold.  Allowing the slow dynamics to depend on the slow variable $y$ means that $g_y$ is no longer identically zero.  Hence, the bifurcation point might not occur at the corner as shown in parts (i) and (ii) of Theorem \ref{thm:corner}.  However, as in part (iii) of Theorem \ref{thm:corner}, the bifurcation point may still remain at the corner even if $g_y(0,0,) \neq 0$, depending on the values of $f_+'(0)$ and $f_-'(0)$.  This is a contrast to the analogous smooth Hopf bifurcation, where the bifurcation occurs at the extremum only if $g_y = 0$ at the fold.  

At the smooth fold, the bifurcation is a standard, smooth Hopf bifurcation.  The classical canard theory is holds up until the grazing bifurcation, and is extended by comparison with a shadow system.  Again, comparison with a shadow system is utilized to discuss global dynamics when the bifurcation occurs near the corner.  However, at the corner there are five distinct possible bifurcations.  Each bifurcation, and whether or not it produces stable canard cycles, is determined by the type of equilibrium on each side of the splitting line.
	\begin{enumerate}
		\item The bifurcation is smooth Hopf bifurcation, and criticality is determined in the standard way for smooth systems.  A smooth Hopf bifurcation happens if and only if the bifurcation does not occur as the equilibrium passes through the splitting line.  As such, the type of equilibrium does not change across the splitting line.  
		\item A stable focus becomes an unstable focus upon passing through the splitting line.  This is often called a `nonsmooth Hopf' bifurcation, and it's criticality is determined by $\Lambda$ in equation \eqref{nsHCrit}.  Stable canard cycles are created if the bifurcation is supercritical.
		\item A stable node becomes an unstable focus upon passing through the splitting line.  I propose that this should be called a `nonsmooth Hopf-like' bifurcation, and this bifurcation is always supercritical, creating stable cycles.
		\item A stable node becomes and unstable node upon passing through the splitting line creating stable relaxation oscillations.  The bifurcation is a super-explosion because the system forgoes the canard explosion, and it is always supercritical. 
		\item A stable focus becomes an unstable node upon passing through the splitting line, and unstable relaxation oscillations occur just before the bifurcation point. The bifurcation is a subcritical super-explosion, and again, no stable canard cycles are created.
	\end{enumerate}
	
Theorem \ref{thm:corner} provides conditions determining which of the bifurcations occurs at the corner.  One remarkable consequence of these conditions is that  if $f_+'(0) > 0$ and for $\eps$ sufficiently small, the system will not have canard orbits.  This is a major difference from the classical smooth theory where canard cycles exist for $\eps$ sufficiently small.  Essentially, there is a range of $\eps$ bounded above 0 for which canard cycles exist.  The reason for the contrast from the smooth case is that for $\eps$ sufficiently small, the bifurcation will be a super-explosion, creating an unstable node.  Because a node is created, there is no rotation near the unstable equilibrium, and it is not possible for an attracting slow manifold to connect to a repelling slow manifold.  Trajectories pass into the right half plane, and are bounded away from a slow manifold by the strong repelling trajectory from the node as in Figure \ref{fig:expBif3}.  However, if the system \eqref{general} is `close enough' to smooth (relative to $\eps$), then the canard explosion will be comparable to the classical smooth case.

The motivation for \cite{aar1}, the predecessor to this work, was to be able to analyze a variation of a nonsmooth thermohaline circulation model \cite{aar2}.  Conceptual climate models have utilized nonsmooth approximations of smooth systems in the past \cite{stommel, Welander82}, and it seems this is becoming ever more popular.  For example, over the last 7 years \cite{abbot11,jormungand,dV07,dV10,Eisenman09,Eisenman12, hogg} all introduce nonsmooth conceptual climate models.  The models in these papers are largely analyzed by simulation.  Further developing our understanding of nonsmooth dynamics will be important in developing our understanding of the climate.

From a more theoretical point of view, this work is indicative of the rejuvenated interest in canards, particularly in nonsmooth systems.  Work is already under way using these results to discuss the existence of sliding canard solutions in discontinuous planar systems, and in certain cases it is essential that the slow dynamics are more complicated than those of a Van der Pol-type system.  Looking forward, it is desirable to obtain results on canards in higher dimensions.  Prohens and Teruel have made the first contribution by obtaining results in piecewise-linear, continuous 3D systems.  By incrementally increasing our understanding of canards in nonsmooth systems, the goal is ultimately to obtain results on the existence nonsmoth mixed-mode oscillations.  

\section*{Acknowledgments}
This work was made possible by support from the Mathematics and Climate Research Network, through NSF grants DMS-0940363 and DMS-1239013.  I would like to thank Anna Barry and the MCRN Nonsmooth Systems group for their conversations.  Lastly, I would like to thank John Guckenheimer and Reviewer \#2 for their comments during the review process.

\bibliographystyle{amsplain}
\bibliography{sources}

\end{document}